%% file: main.tex
\pdfoutput=1
\documentclass[11pt]{article}

\RequirePackage{header}


\input{macros.tex} % Common macros for the project

\begin{document}
\input{tex/title.tex}

\section{Introduction}
%The Schramm-Loewner Evolution introduced by Oded Schramm \cite{Schramm_LERW_UST} has been a central topic in probability as being a conformally invariant random simple curves.
%related mathematical physics for the past 20 years. 
The Schramm-Loewner evolution is a one parameter family of random fractal curves  (denoted as $\operatorname{SLE}_\kappa$ with parameter $\kappa > 0$). It was introduced by Oded Schramm \cite{Schramm_LERW_UST} and has been a central topic in the two dimensional random conformal geometry.
% and arises as the scaling limit of many important conformally invariant two dimensional statistical physics models. 
A version of such curves starting from a fixed boundary point to a fixed interior point on some two-dimensional simply connected domain $D$ are called \emph{radial} $\operatorname{SLE}$s.
Let us recall briefly the definition. 
The radial $\operatorname{SLE}_\kappa$  on the unit disk $\mathbb{D}=\{\zeta\in\mathbb{C}: \left\lvert \zeta \right \rvert=1\}$ targeted at $0$ is the random curve associated to the radial Loewner chain, whose \emph{driving function} $t \mapsto \zeta_t$ is given by a Brownian motion on the unit circle $S^1=\{\zeta\in\mathbb{C}: \left\lvert \zeta \right \rvert=1\}$ with variance $\kappa$. That is,
\begin{equation}\zeta_t := \bm_t := e^{iW_{\kappa t}}, \label{def:circ-bm}\end{equation}
where $W_t$ is a standard linear Brownian motion.
More precisely, we consider the Loewner ODE  for all $z \in \mathbb{D}$
\begin{equation}
\partial_t g_t(z) = -g_t(z)\frac{g_t(z)+\zeta_t}{g_t(z)-\zeta_t},  \label{eq:sle_ODE}
\end{equation}
or equivalently, the Loewner PDE satisfied by $f_t := g_t^{-1}$
\begin{equation} \label{eq:sle_PDE}
\partial_t f_t(z) = zf'_t(z)\frac{z+ \zeta_t}{z-\zeta_t}, \quad z \in \mathbb{D}
\end{equation} 
with the initial 
%boundary 
condition $f_0(z) = g_0(z)=z$. For a given $t > 0$, $f_t$ is a conformal map from $\mathbb{D}$ onto a simply connected domain $D_t \subset \mathbb{D}$ (and $s\mapsto g_s(z)$ is  a well-defined solution of~\eqref{eq:sle_ODE} up to time $t$ if and only if $z \in D_t$)  such that $f_t(0) = 0$ and $f_t'(0) = e^{-t}$.
The family of conformal maps $\{f_t\}_{t\ge 0}$ is called the \emph{capacity parametrized radial Loewner chain} or \emph{normalized subordination chain} driven by $t\mapsto \zeta_t$. $\operatorname{SLE}_\kappa$ is the curve  $t \mapsto \gamma_t$ which is defined as
$\gamma_t:=\lim_{r\to 1-}f_t(r\zeta_t)$, see \cite{RS_Basic}. In particular, the curve starts at $\gamma_0 = 1$. 
%The associated family of \emph{Loewner compact hulls} is $\{K_t : = \m D \setminus D_t = \g[0,t]\}$.
The radial $\operatorname{SLE}_\kappa$ on an arbitrary simply connected domain $D$ is defined via the unique conformal map from $\mathbb{D}$ to $D$ respecting the starting and target points.
It is well-known that $\operatorname{SLE}_\kappa$ exhibits phase transitions as $\kappa$ varies. Larger values of $\kappa$ correspond in some sense to ``wilder'' $\operatorname{SLE}_\kappa$ curves; in the $\kappa \geq 8$ regime the curve is space-filling. %and the direct reversibility statements break down.

In this work, we study the $\kappa \to \infty$ asymptotic behavior of radial $\operatorname{SLE}$. 
%It should be noted that these radial SLE curves for large $\k$ are not really relevant for statistical mechanics models, but as we shall explain below, the motivation for this study rather comes from complex analysis and geometric measure theory
To simplify notation we consider $\operatorname{SLE}_\kappa$ run on the time interval $[0,1]$ throughout the paper, but our results are easily generalized to arbitrarily bounded time intervals.
Hence we denote by $\{\cdot\}$ the family $\{\cdot\}_{t \in [0,1]}$ to avoid repeating indices.

Our first result (Proposition~\ref{prop-LLN-main}) 
characterizes the limit as $\kappa \to \infty$ of the time-evolution of the $\operatorname{SLE}_\kappa$ hulls.
%, i.e. $t\mapsto \g_{[0,t]}$, converges to a deterministic limit $t \mapsto \mathbb D \backslash e^{-t} \mathbb D$. 
We argue heuristically as follows. We view the time-dependent vector field  $\{-z (z+ \zeta_t)/(z-\zeta_t)\}$ which generates the flow $\{g_t\}$  as $\{\int_{S^1} -z (z+ \zeta)/(z-\zeta) \delta_{\bm_t}(\zeta)\}$, where $\delta_{\bm_t}$ is the Dirac mass at $\bm_t$.  During a short time interval where the flow is well-defined for the point $z$, we have $g_t(z) \approx g_{t+\Delta t}(z)$ and hence
\begin{align*}
\Delta g_t (z) &\approx \int_{t}^{t+\Delta t} \int_{S^1} -g_t (z) (g_t (z)+ \zeta)/(g_t (z)-\zeta) \delta_{\bm_s}(\zeta) ds \\
&= \int_{S^1} -g_t (z) (g_t (z)+ \zeta)/(g_t (z)-\zeta) d(\oc_{t+\Delta t} (\zeta)- \oc_{t}(\zeta)),    
\end{align*}
where $\oc_t$ is the \emph{occupation measure} (or \emph{local time}) on $S^1$ of $\bm$ up to time $t$. 
We show that as $\kappa \to \infty$, the driving function oscillates so quickly that its local time in $[t,t+\Delta t]$ is almost uniform on $S^1$, so in the limit we get a \emph{measure-driven Loewner chain} with driving measure  uniform on $S^1$. That is,
$$\partial_t g_t (z) = \frac{1}{2\pi} \int_{S^1}-g_t(z)\frac{g_t(z)+\zeta}{g_t(z)-\zeta} d\zeta, $$
where $d\zeta$ denotes the Lebesgue measure. This implies $\partial_t g_t(z) =  g_t(z)$, that is, $g_t (z) = e^{t} z$ or equivalently $f_t (z) = e^{-t} z$. See Section~\ref{sec:msr-loewner} for more details on the measure-driven Loewner chain. We show in Section~\ref{section-LLN}:
\begin{proposition}\label{prop-LLN-main}
	As $\kappa \to \infty$, the Loewner chain $\{f_t\}$ driven by $\{\zeta_t\}$ \textnormal{(}defined in \eqref{def:circ-bm}\textnormal{)} converges to $\{z \mapsto e^{-t} z\}$ almost surely, with respect to the uniform Carath\'eodory topology.
\end{proposition}
We shall mention that Loewner chains are also used in the study of the Hastings-Levitov model of randomly aggregating particles and similar small-particle limits have been studied, see  \cite{JohSolTur2012} and references therein.

The heuristic argument above suggests that the large deviations of $\operatorname{SLE}_\kappa$ boil down to the large deviations of the Brownian occupation measure, which we now describe. 

For any metric space $X$, let 
$\mathcal{M}_1(X)$ denote the set of Borel probability measures equipped with the Prokhorov topology (the topology of weak convergence).  
Let
\begin{equation}\label{eq-def-N}
\mathcal{N} = \{\rho\in\mathcal{M}_1(S^1 \times [0,1]): \rho(S^1\times I)= |I| \text{ for all intervals } I \subset [0,1]\}.
\end{equation}
The condition imposed here allows us to write $\rho \in \mathcal{N}$ as a disintegration $\{ \rho_t\}$ over the time interval $[0,1]$ (with $\rho_t \in \mathcal{M}_1(S^1)$ for a.e. $t$); see \eqref{def:disint}. We identify $\rho$ and the time-indexed family $\{\rho_t\}$. The second result we show is:

\begin{theorem}
	\label{thm:ldp-circ-bm}
	The process of measures $ \{\delta_{\bm_t}\}\in\mathcal{N}$ satisfies the large deviation principle with good rate function $\mathcal{E} (\rho) :=\int_0^1 I(\rho_t)\, dt$ for $\rho\in\mathcal{N}$, where $I(\mu)$ is defined for each $\mu\in\mathcal{M}_1(S^1)$ as
	\begin{equation}\label{eq:ldp_local}
	I(\mu) :=  
	\frac 12\int_{S^1} |\phi'(\zeta)|^2  d\zeta  
	\end{equation}
	if $\mu(d\zeta)=\phi^2 (\zeta)\,d\zeta$  and  $\phi$  is absolutely continuous, and $I(\mu) : =  \infty$ otherwise.
	That is, for every closed set $C$ and open set $G$ of $\mathcal{N}$,
	\begin{align*}
	\limsup_{\kappa \to \infty} & \frac{1}{\kappa} \log \mathbb P\left[\{\delta_{\bm_t}\} \in C\right] \le - \inf_{\rho \in C} \mathcal{E}(\rho); \\
	\liminf_{\kappa \to \infty} & \frac{1}{\kappa} \log \mathbb P\left[\{\delta_{\bm_t}\} \in G \right] \ge - \inf_{\rho \in G} \mathcal{E}(\rho);
	\end{align*}
	and the sub-level set $\{\rho \in \mathcal{N} : \mathcal{E}(\rho) \le c\}$ is compact for all $c > 0$.
\end{theorem}

Our proof is based on a result by Donsker and Varadhan \cite{DonVar1975} on large deviations of the Brownian occupation measure (see Sections~\ref{sec:ldp_local_time}--\ref{sec:proof_1.2}).
The $\kappa \to \infty$ large deviations of $\operatorname{SLE}$
then follows immediately from the continuity of the Loewner transform (Theorem~\ref{thm:cont-bij-loewner-transf})
and the contraction principle \cite[Theorem 4.2.1]{DemZei2010}.
\begin{corollary}
	\label{cor:ldp-sle}
	The family of $\operatorname{SLE}_{\kappa}$ satisfies the $\kappa \to \infty$ large deviation principle with the good rate function
	$$I_{\operatorname{SLE}_{\infty}} (\{K_t\}) :=  \mathcal{E}(\rho),$$
	where $\{\rho_t\}$ is the driving measure whose Loewner transform is $\{K_t\}$.
\end{corollary}

Let us conclude the introduction with two comments.

The study of large deviations of $\operatorname{SLE}$, while of inherent interest, is also motivated by problems from complex analysis and geometric function theory. In a forthcoming work \cite{VW2}, Viklund and the third author investigate the
\emph{duality} between the rate functions of $\operatorname{SLE}_{0+}$ (termed as the \emph{Loewner energy} introduced in \cite{W1,RW}) and  $\operatorname{SLE}_\infty$ that is reminiscent of the $\operatorname{SLE}$ duality \cite{Dubedat09Duality,Zhan08Duality} 
which couples $\operatorname{SLE}_{\kappa}$ to the outer boundary of $\operatorname{SLE}_{16/\kappa}$ for $\kappa < 4$.
Note that $\mathcal{E} (\rho)$ attains its minimum if and only if $\{D_t\}$ are concentric disks, and $\{\partial D_t\}$ are circles which also have the minimal Loewner energy.  

It is also natural to consider the large deviations of \emph{chordal} $\operatorname{SLE}_{\infty}$ (say, in $\mathbb{H}$ targeted at $\infty$). 
However, in contrast with the radial case, the family indexed by $\kappa$ of random measures $\{\delta_{W_{\kappa t}}\}$ on $\mathbb{R} \times [0,1]$ is not tight and the corresponding Loewner flow converges to the identity map for any fixed time $t$. To obtain a non-trivial limit, one needs to renormalize appropriately (see e.g., Beffara's thesis \cite[Sec.5.2]{beffara} for a non-conformal normalization) and consider generalized Stieltjes transformation of measures for the large deviations.
Therefore, for simplicity we choose to study the radial case and suggest the large deviations of chordal $\operatorname{SLE}_{\infty}$ as an interesting question.
We will show a simulation of large-$\kappa$ chordal SLEs and discuss some other questions at the end of the paper.

The paper is organized as follows: In Section~\ref{sec:msr-loewner}, we explain the measure-driven radial Loewner evolution. 
In Section~\ref{sec:circ-bm} we prove the main results of our paper. In Section~\ref{section-comments} we present some comments, observations and questions.

\bigskip
{\bf Acknowledgements: } We thank Wendelin Werner for helpful comments on the presentation of our paper and Fredrik Viklund and Scott Sheffield for inspiring discussions. We also thank an anonymous referee
for comments on the first version of this paper. M.A. and M.P.\ were partially supported by NSF Award DMS 1712862. Y.W.\ was partially supported by the NSF Award DMS 1953945.

\section{Measure-driven radial Loewner evolution} \label{sec:msr-loewner}
In this section we collect some known facts on the measure-driven Loewner evolution (also known as Loewner-Kufarev evolution) that are essential to our proofs.
Recall that $$\mathcal{N} = \{\rho\in\mathcal{M}_1(S^1 \times [0,1]): \rho(S^1\times I)= |I| \text{ for all intervals } I \subset [0,1]\}$$
endowed with the Prokhorov topology.
%That is, $\rho^n \to \rho$ in $\X$ if and only if for every continuous function 
% $\varphi$ on $S^1\times [0,1]$ we have 
% $$\int_{S^1\times [0,1]} \varphi(\z,t)d\rho^n \to \int_{S^1\times [0,1]} \varphi(\z,t)d\rho.$$
From the disintegration theorem (see e.g. \cite[Theorem 33.3]{billingsley}),
%~\cite[Theorem 1]{ChaPol1997}
for each measure $\rho\in\mathcal{N}$ there exists a Borel measurable map $t\mapsto \rho_t$ (sending $[0,1] \to \mathcal{M}_1(S^1)$) such that for every measurable function $\varphi:S^1 \times [0,1] \rightarrow \mathbb{R}$ we have
\begin{equation}\label{def:disint}
\int_{S^1\times [0,1]}\varphi(\zeta,t)\,d\rho = \int_0^1\int_{S^1} \varphi(\zeta,t)\,\rho_t(d\zeta)\,dt.
\end{equation}
We say $\{\rho_t\}$ is a \emph{disintegration} of $\rho$; it is unique in the sense that any two disintegrations $\{\rho_t\}, \{ \widetilde \rho_t\}$ of $\rho$ must satisfy $\rho_t = \widetilde \rho_t$ for a.e. $t$. We always denote by $\{\rho_t\}$ one such disintegration of $\rho\in\mathcal{N}$.

The Loewner chain driven by a measure $\rho\in\mathcal{N}$ is defined similarly to \eqref{eq:sle_ODE}. For $z \in \mathbb{D}$, consider the \emph{Loewner-Kufarev ODE}
\[\partial_t g_t(z) = -g_t(z) \int_{S^1} \frac{g_t(z) + \zeta}{g_t (z) - \zeta} \,\rho_t(d\zeta)\]
with the initial condition $g_0(z)=z$. Let $T_z$ be the supremum of all $t$ such that the solution is well-defined up to time $t$ with $g_t(z)\in\mathbb{D}$, and $D_t:=\{z\in \mathbb{D}: T_z >t\}$ is a simply connected open set containing $0$. 
We define the \emph{hull} $K_t:=\mathbb{D}\setminus D_t$ associated with the Loewner chain.
%; this is a %compact 
%subset of $\mathbb D$ with simply connected complement. 
Note that when $\kappa \ge 8$, the family $\{\gamma[0,t]\}$ of radial $\operatorname{SLE}_\kappa$
%driven by $\{\bm_t\}$
is exactly the family of hulls $\{K_t\}$ driven by the measure $\{\delta_{\bm_t}\}$.

The function $g_t$ defined above is the unique conformal map of $D_t$ onto $\mathbb{D}$ such that $g_t (0) = 0$ and $g_t'(0) > 0$; moreover $g_t'(0) = e^t$ (i.e. $D_t$ has conformal radius $e^{-t}$ seen from $0$)
since 
$\partial_t \log g_t'(0) =  |\rho_t| = 1$ (see e.g. \cite[Thm.~4.13]{Lawler_book}).

If $g_t$ is the solution of a Loewner-Kufarev ODE then its inverse $f_t=g_t^{-1}$ satisfies the \emph{Loewner-Kufarev PDE}:
$$\partial_t f_t(z) = z f_t'(z) \int_{S^1} \frac{z + \zeta}{z - \zeta} \rho_t(d\zeta),$$
and vice versa.
Note that $f_t (0) = 0$, $f_t'(0) = e^{-t}$, and $f_t(\mathbb{D}) = D_t \subset f_s(\mathbb{D})$ for $s \le t$. Such a time-indexed family $\{f_t\}$ is called \emph{a normalized chain of subordinations}. 
We write $\mathcal{S}$ for the set of normalized chains of subordinations $\{f_t\}$ on $[0,1]$.
%Note that 
An element of $\mathcal{S}$ can be equivalently represented by either $\{f_t\}$ or the process of hulls $\{K_t\}$. %We will not distinguish the three representations in the sequel.
The map $\mathcal{L}:\rho\mapsto \{f_t\}$ (or interchangeably $\mathcal{L}:  \rho\mapsto \{K_t\}$) is called the \emph{Loewner transform}.
In fact, $\mathcal{L}$ is a bijection:
%; this is an easy consequence of the following theorem.
\begin{theorem}[Bijectivity of the Loewner transform~{\cite[Satz 4]{Pom1965}}]
	The family $(f_t)_{t\in [0,1]}$ is a normalized chain of subordination over $[0,1]$ if and only if 
	\begin{itemize}
		\item $t \mapsto f_t (z)$ is absolutely continuous in $[0,1]$ and for all $r <1$, there is $K(r)$ > 0 such that $|f_t(z)-f_s(z)| \le K(r)|t-s|$ for all $z \in r\mathbb{D}$;
		\item and there is a ($t$-a.e. unique) function $h(z,t)$ that is analytic in $z$, measurable in $t$ with $h(0,t) = 1$ and $\operatorname{Re} h (z,t)> 0$, so that for  $t$-a.e. we have
		$$\partial_t f_t(z) = -z  f_t'(z) h(z,t).$$
	\end{itemize}
\end{theorem}
From the Herglotz representation of $h (\cdot, t)$, there exists a unique $\rho_t \in \mathcal{M}_1(S^1)$ such that 
$$h(z,t) = \int_{S^1} \frac{\zeta+z}{ \zeta-z} \rho_t(d\zeta), \quad \forall z\in \mathbb{D}.$$
Therefore $\{f_t\}$ satisfies the Loewner PDE driven by the (a.e. uniquely determined) measurable function $t \mapsto \rho_t$.

We now equip $\mathcal{S}$ with a topology. View $\mathcal{S}$ as the set of normalized chains of subordinations $\{f_t\}$ on $[0,1]$, and change notation by writing $f(z,t) = f_t(z)$. We endow $\mathcal{S}$ with the topology of uniform convergence of $f$ on compact subsets of $\mathbb{D} \times [0,1]$. (Equivalently, if we view $\mathcal{S}$ as the set of processes of hulls $\{K_t\}$, this is the topology of \emph{uniform Carath\'eodory convergence}.)
The continuity of $\mathcal{L}$ has been, e.g., derived in \cite[Proposition 6.1]{MilShe2016} (see also \cite{JohSolTur2012}).
\begin{theorem}[Continuity]\label{thm:cont-bij-loewner-transf}
	The Loewner transform $\mathcal{L}: \mathcal{N} \to \mathcal{S}$ is a homeomorphism. 
\end{theorem}
\section{Proofs of the main results}\label{sec:circ-bm}
In this section, we study the random measure $\{\delta_{\bm_t}\}\in \mathcal{N}$.
In Section~\ref{section-approx} we approximate $\mathcal{N}$ by spaces of time-averaged measures. In Section~\ref{section-LLN} we verify that $\{\delta_{\bm_t}\}\in \mathcal{N}$ converges
almost surely
as $\kappa \to \infty$ to the uniform measure on $S^1 \times [0,1]$; this yields Proposition~\ref{prop-LLN-main}. In Section~\ref{sec:ldp_local_time}, we review the large deviation principle for the circular Brownian motion occupation measure, which is a special case of seminal work of Donsker and Varadhan \cite{DonVar1975}. Finally, in Section~\ref{sec:proof_1.2} we prove Theorem~\ref{thm:ldp-circ-bm}, the large deviation principle for $\{\delta_{\bm_t}\}\in \mathcal{N}$.

\subsection{Time-discretized approximations of measures}\label{section-approx}
%Recall from~\eqref{eq-def-N} that $\X \subset \M_1(S^1\times [0,1])$ is the set of probability measures on $S^1\times [0,1]$ whose marginal on $[0,1]$ is given by $\Leb_{[0,1]}$, equipped with the induced Prokhorov topology.
%In this section, we discuss approximations of measures $\rho \in \X$ and express $\X$ as a projective limit of such approximations. 
We emphasize that the results of this section are wholly deterministic.

For $n \geq 0$, let $\mathcal{I}_n:=\{0,1,2,\cdots, 2^n-1\}$ be an index set, and define
$\mathcal Y_n := \left(\mathcal{M}_1(S^1)\right)^{\mathcal{I}_n}.$
We note that $\mathcal{Y}_n$ is endowed with the product topology. For each $i \in \mathcal{I}_n$ we define a function $P_n^i: \mathcal{N} \to \mathcal{M}_1(S^1)$ via
\begin{equation}\label{eq-def-proj}
P_n^i(\rho) := 2^n \int_{i/2^n}^{(i+1)/2^n} \rho_t \ dt,
\end{equation}
where here $\{\rho_t\}$ is a disintegration of $\rho$ with respect to $t$, as in \eqref{def:disint}. We define also the map $P_n: \mathcal{N} \to \mathcal{Y}_n$ via $P_n = (P_n^i)_{i \in \mathcal{I}_n}$. That is, $P_n$ averages $\rho$ along each $2^{-n}$-time interval, and outputs the $2^n$-tuple of these $2^n$ time-averages.

We consider $\mathcal{Y}_n$ to be the space of time-discretized approximations of $\mathcal{N}$, in the following sense. Define a map $F_n: \mathcal{Y}_n \to \mathcal{N}$ via
\[F_n\left( ( \mu_i)_{i \in \mathcal{I}_n}\right) := \sum_{i \in \mathcal{I}_n} \mu_i \otimes \operatorname{Leb}_{[i/2^n, (i+1)/2^n]}. \]
Then one can view $F_n(P_n(\rho))$ as a ``level-$n$ approximation'' of $\rho$ (see Lemma~\ref{lem-projection}).
%, in the sense that $F_n(P_n(\rho))$ converges in the Prokhorov topology to $\rho$ as $n \to \infty$.

We have provided a way of projecting an element of $\mathcal{N}$ to the space of level-$n$ approximations $\mathcal{Y}_n$. Now we write down a map $P_{n,n+1}: \mathcal{Y}_{n+1} \to \mathcal{Y}_n$ which takes in a finer approximation and outputs a coarser approximation:
\[P_{n, n+1} \left((\mu_i)_{i\in \mathcal{I}_{n+1}} \right):= \left(\frac{\mu_0+\mu_1}2, \dots, \frac{\mu_{2^{n+1}-2} + \mu_{2^{n+1} -1}}2 \right).\]
That is, we average pairs of components of $\mathcal{Y}_{n+1}$. It is clear that
\begin{equation}\label{eq-projection}
P_n = P_{n,n+1} \circ P_{n+1}.
\end{equation}

%As $n \to \infty$, the space of approximations $\mc Y_n$ converges in some sense to $\X$. 
The convergence of $P_n(\rho_j) \xrightarrow{j \to \infty} P_n(\rho)$ in $\mathcal{Y}_n$ is equivalent to the convergence $\rho_j(f) \xrightarrow{j \to \infty} \rho(f)$ for the functions $f$ which are piecewise constant in time for each time interval $(i/2^n, (i+1)/2^n)$. For each fixed $n$, this is a coarser topology than that of $\mathcal{N}$.
%these ``piecewise constant in time'' functions can approximate arbitrarily closely any element in $C(S^1 \times [0,1])$. 
The following lemma shows that the $n \to \infty$ topology agrees with that of $\mathcal{N}$.
%We rigorously state this in Lemma~\ref{lem-projection}.
\begin{lemma}\label{lem-projection}
	We have
	$\mathcal{N} = \varprojlim \mathcal{Y}_n. $
	That is, as topological spaces, $\mathcal{N}$ is the projective (inverse) limit of $\mathcal{Y}_n$ as $n \to \infty$.
\end{lemma}

\begin{proof}
	Let $\mathcal{Y} = \varprojlim \mathcal{Y}_n$; this is the subset of $\prod_{j=0}^\infty \mathcal{Y}_j$ comprising elements $(y^0, y^1, \dots)$ such that $P_{n,n+1}(y^{n+1}) = y^n$ for all $n\geq0$. The topology on $\mathcal{Y}$ is inherited from $\prod_{j=0}^\infty \mathcal{Y}_j$. 
	%We will construct a homeomorphism from $\X$ to $\mc Y$.
	Because of the coherence relation~\eqref{eq-projection}, we can define a map $P: \mathcal{N} \to \mathcal{Y}$ by
	$P(\rho) := (P_j(\rho))_{j\geq0}. $
	We now show that $P$ is a homeomorphism.
	\medskip
	
	\noindent\textbf{Showing that $P$ is continuous.}
	Since the topology on $\mathcal{Y}$ is inherited from the product topology on $\prod_{n=0}^\infty \mathcal{Y}_j$, it suffices to show that the map $P: \mathcal{N} \to \prod_{n=0}^\infty \mathcal{Y}_n$ is continuous, i.e. $P_n : \mathcal{N} \to \mathcal{Y}_n$ is continuous for each $n$. But this is clear: if two measures in $\mathcal{N}$ are close in the Prokhorov topology, then so is the time-average of these measures on a time interval. %(The discontinuous cut-off on the time interval raises no issues because measures in $\X$ have marginals on $[0,1]$ given by Lebesgue measure.)
	\medskip
	
	\noindent\textbf{Showing that $P$ is a bijection.} Fix $f \in C(S^1 \times [0,1])$. We claim that for any $\varepsilon > 0$, there exists $n_0 = n_0(f,\varepsilon)$ such that for all $m,n \geq n_0$ and $y = (y^0, y^1, \dots) \in \mathcal{Y}$ we have
	\begin{equation} \label{eq-smear}
	\left| (F_m(y^m))(f) - (F_n(y^n))(f) \right| < \varepsilon.
	\end{equation}
	To that end we note that $f$ is uniformly continuous; we can choose $\delta>0$ so that $|f(\zeta, t) - f(\zeta, t')|< \varepsilon$ whenever $|t - t'|<\delta$. Choosing $n_0$ such that $2^{-n_0} < \delta$, we obtain~\eqref{eq-smear}.
	
	Now we show that $P$ is a bijection. By~\eqref{eq-smear}, for each $y = (y^0, y^1, \dots)\in \mathcal{Y}$ we can define a bounded linear functional $T_y: C(S^1 \times [0,1]) \to \mathbb{R}$ via
	\[
	T_y(f) = \lim_{n \to \infty} (F_n(y^n))(f) \quad \text{ for }y = (y^0, y^1, \dots).
	\]
	Clearly $T_y$ maps nonnegative functions to nonnegative reals, so the Riesz-Markov-Kakutani representation theorem tells us there is a unique\footnote{The representation theorem yields a unique \emph{regular} Borel measure, but since $S^1 \times [0,1]$ is compact, \emph{all} Borel measures on $S^1 \times [0,1]$ are regular.} Borel measure $\rho$ on $S^1 \times [0,1]$ such that $\rho(f) = T_y(f)$ for all $f \in C^1(S^1 \times [0,1])$; it is easy to check that $\rho \in \mathcal{N}$. Thus, for each $y\in \mathcal{Y}$ the equation $P(\rho) = y$ has a unique solution in $\mathcal{N}$, so $P$ is a bijection.
	\medskip
	
	\noindent\textbf{Concluding that $P$ is a homeomorphism.} We note that $\mathcal{N}$ is compact (since $S^1 \times [0,1]$ is compact) and $\mathcal{Y}$ is compact (since each $\mathcal{Y}_n$ is compact and Hausdorff). Since $P$ is a continuous bijection of compact sets, it is a homeomorphism.
	%
	%\noindent\textbf{Showing that $P^{-1}: \mc Y \to \X$ is continuous.}
	%This is equivalent to the statement that for any sequence $(y_k)_{k \geq 0} \subset \mc Y$ converging to $y_\infty$ and any continuous function $f \in C(S^1 \times [0,1])$, we have
	%\[\lim_{k\to\infty} T_{y_k}(f) = T_{y_\infty}(f).\]
	%Fix $\eps > 0$. By~\eqref{eq-smear}, we see that for all large $n$ we have
	%\[\left| (F_n(y^n_k))(f) - T_{y_k}(f) \right| < \eps \quad \text{ for all }k. \]
	%Since $\lim_{k \to \infty} y^n_k = y^n_\infty$, for all large $k$ we have $|(F_n(y^n_k))(f) - (F_n(y^n_\infty))(f)| < \eps$, and hence  $|T_{y_k}(f) - T_{y_\infty}(f)| < 3\eps$. Thus $P^{-1}$ is continuous.
\end{proof}

\subsection{Almost sure limit of SLE driving measures}\label{section-LLN}

Consider a Brownian motion $\bm_t$ on the unit circle $S^1=\{\zeta\in\mathbb{C}:\left\lvert \zeta \right \rvert=1\}$ started at $1$ with variance $\kappa$ as \eqref{def:circ-bm}. %or diffusion rate $\k$.
%, that is
% \begin{equation}%\label{eq-coupling}
% \bm_t := e^{iW_{\k t}}
% \end{equation}
%where $W_t$ is a standard linear Brownian motion. 
%We study $\bm_t$ as the driving function of the radial Loewner equation \eqref{eq:sle_ODE} defining $\SLE_\k$.
Define the \emph{occupation measure} of $\bm_t$: 
\begin{equation*}
\oc_t(A) = \int_0^t \mathbf{1}\{\bm_s\in A\}\, ds\quad \text{ for Borel sets }A\subset S^1.
\end{equation*}
%where $\mc B(S^1)$ is the Borel set of $S^1$.
Let $\aoc_t = t^{-1} \oc_t$ be the \emph{average occupation measure} of $\bm$ at time $t$ (its normalization gives $\aoc_t\in \mathcal{M}_1(S^1)$).
An easy consequence of the ergodic theorem is the following almost sure $t\to \infty$ limit of $\aoc[1]_t$; we include the proof for completeness.
\begin{lemma}\label{lem-LLN}
	Almost surely, as $t \to \infty$ we have $\aoc[1]_t \to (2\pi)^{-1} \operatorname{Leb}_{S^1}$ in $\mathcal{M}(S^1)$.
\end{lemma}
\begin{proof}
	It suffices to show that for any continuous function $f: S^1 \to \mathbb{R}$ we have almost surely
	\begin{equation}\label{eq-LLN}
	\lim_{t \to \infty} \frac1t \int_0^t f(\bm[1]_s) \ ds = \frac1{2\pi} \int_{S^1} f(\zeta) \ d \zeta.
	\end{equation}
	Given this, by choosing a suitable countable collection of functions, we obtain the lemma.
	
	%Let $(\Omega, \mathbb P)$ be a probability space (we suppress the $\sigma$-algebra) such that $\{\bm[1]_t(\omega)\}$ is Brownian motion started at $1$.
	Let $(\Omega, \mathbb{P})$ be a Wiener space so that $\mathbb{P}$ is the law of $\bm[1]_t$.
	Consider the expanded probability space given by $(\Omega \times S^1, \mathbb P \otimes (2\pi)^{-1} \operatorname{Leb}_{S^1})$, and let $(\omega, e^{i\theta})$ correspond to the random path $e^{i\theta} \bm[1]_t(\omega)$. That is, after sampling an instance of Brownian motion $\bm[1]_t(\omega)$ started at $1$, we apply an independent uniform rotation to the circle so the Brownian motion starts at $e^{i\theta}$ instead. A consequence of Birkhoff's ergodic theorem is that for a.e. $(\omega, e^{i\theta}) \in \Omega \times S^1$, we have
	\begin{equation}
	\label{eq-ergodic}
	\lim_{t\to\infty} \frac1t \int_0^t f(e^{i\theta} \bm[1]_s(\omega)) \ ds = \frac1{2\pi} \int_{S^1} f(\zeta) \ d\zeta.
	\end{equation}
	Equivalently, for a.e. $e^{i\theta} \in S^1$, we have~\eqref{eq-ergodic} for a.e. $\omega$. Taking a sequence of $e^{i\theta}$ converging to $1$ and using the uniform continuity of $f$, we obtain~\eqref{eq-LLN}. This concludes the proof of Lemma~\ref{lem-LLN}.
\end{proof}

Now we justify the heuristic argument in the introduction, which said that as $\kappa\to\infty$, the Brownian motion $\bm_t$ moves so quickly that the driving measure converges to $(2\pi)^{-1}\operatorname{Leb}_{S^1}\otimes \operatorname{Leb}_{[0,1]}$.
\begin{lemma}\label{lem-BM-conv-to-flat}
	%Define $\bm_t$ via~\eqref{eq-coupling}. 
	As $\kappa \to \infty$,
	$\{\delta_{\bm_t}\}$ converges almost surely to $(2\pi)^{-1} \operatorname{Leb}_{S^1} \otimes \operatorname{Leb}_{[0,1]}$ in $\mathcal{N}$.
\end{lemma}
\begin{proof}
	Lemma~\ref{lem-projection} states that $\mathcal{N}$ is the projective limit of the spaces $\mathcal{Y}_n$ defined in Section~\ref{section-approx}, with projection map from $\mathcal{N}$ to $\mathcal{Y}_n$ given by $(P^i_n)_{i \in \mathcal{I}_n}$. %Each map $P^i_n$ (defined in~\eqref{eq-def-proj}) maps $\rho \in \X$ to the element of $\M_1(S^1)$ given by the average of $\rho$ on the time interval $[i/2^n, (i+1)/2^n]$.
	It thus suffices to show that as $\kappa \to \infty$, the random measure $P^i_n(\{\delta_{\bm_t}\})$  converges almost surely  to $P^i_n((2\pi)^{-1} \operatorname{Leb}_{S^1} \otimes \operatorname{Leb}_{[0,1]}) = (2\pi)^{-1}\operatorname{Leb}_{S^1}$ in the Prokhorov topology, namely,
	\[\lim_{\kappa \to \infty} 2^n \int_{i/2^n}^{(i+1)/2^n} \delta_{\bm_t} \ dt = %(2\pi)^{-1}
	\frac{1}{2\pi}\operatorname{Leb}_{S^1}. 
	%\quad \text{in the Prokhorov topology}.
	\]
	This is true since Lemma~\ref{lem-LLN} tells us that almost surely%, in the Prokhorov topology we have
	\[
	\lim_{\kappa \to \infty} \frac{2^n}{i} %i^{-1}2^n 
	\int_{0}^{i/2^n} \delta_{\bm_t} \ dt = %(2\pi)^{-1} \Leb_{S^1} =
	\lim_{\kappa \to \infty} \frac{2^n}{i+1} \int_{0}^{(i+1)/2^n} \delta_{\bm_t} \ dt =  \frac{1}{2\pi} \operatorname{Leb}_{S^1}.
	\]
	Hence Lemma~\ref{lem-BM-conv-to-flat} holds.
\end{proof}

%The proof of  
\begin{proof}[Proof of Proposition~\ref{prop-LLN-main}] It follows immediately from Theorem~\ref{thm:cont-bij-loewner-transf} and Lemma~\ref{lem-BM-conv-to-flat}. 
\end{proof}
%The rest of this section is devoted to the proof of Theorem~\ref{thm:ldp-circ-bm}.

%We thus obtain Proposition~\ref{prop-LLN-main}, an almost sure description of the $\k\to\infty$ limit of $\SLE_\k$.

% \begin{proof}[Proof of Proposition~\ref{prop-LLN-main}]
% The proof follows immediately from Theorem~\ref{thm:cont-bij-loewner-transf} and Lemma~\ref{lem-BM-conv-to-flat}.
% \end{proof}

\subsection{Large deviation principle of occupation measures}\label{sec:ldp_local_time}
In this section, we discuss the large deviation principle of Brownian motion occupation measures on $S^1$ as $\kappa \to \infty$. %This is a special case of \cite{DonVar1975}.% which then imply the large deviations of $\SLE_{\k}$ (Corollary~\ref{cor:ldp-sle}).

Recall that $\aoc_t = t^{-1} \oc_t$ is the average occupation measure of $\bm$ at time $t$. By Brownian scaling we have that (recall that the upper index is diffusivity and the lower index is time) $\aoc[1]_{\kappa t}$ and $\aoc_t$ equal in law, so it suffices to understand the large deviation principle for $\aoc[1]_{t}$ as $t\to\infty$. This follows from a more general result of Donsker and Varadhan;
we state the result for Brownian motion on $S^1$.

%For $\mu \in \mc M_1(S^1)$, define the function \[I(\mu) = -\inf_u \int_{S^1} \frac{u''(x)}{2u(x)}\,\mu(dx), \]
%where the infimum is taken over all functions $u: S^1 \to \mathbb R$ which are strictly positive and twice-differentiable.

\begin{theorem}[{\cite[Theorem 3]{DonVar1975}}]\label{thm-DV}
	Define $\tilde{I}: \mathcal{M}_1(S^1) \to \mathbb{R}_{\ge 0}$ by
	\begin{equation}\label{eq:ldp_local_inf}
	\tilde{I}(\mu) : = -\inf_{ u > 0, \, u \in C^2(S^1)} \int_{S^1} \frac{u''}{2u}(\zeta) \mu (d\zeta) = -\inf_{ u > 0, \, u \in C^2(S^1)} \int_{S^1} \frac{L(u)}{u}(\zeta) \mu (d\zeta),
	\end{equation}
	where $L(u) = u''/2$ is the infinitesimal generator of the Brownian motion on $S^1$.
	The average occupation measure $\aoc_1$ admits a large deviation principle as $\kappa \to \infty$, with rate function $\tilde{I}$.
	That is, for any closed set $C \subset \mathcal M_1(S^1)$,
	\begin{equation}\label{eq:ldp_local_upper}
	\limsup_{\kappa \to \infty} \frac{1}{\kappa} \log \mathbb{P} [ \aoc_1 \in C ] \leq -\inf_{\mu \in C} \tilde{I}(\mu),
	\end{equation}
	and for any open set $G \in \mathcal M_1(S^1)$,
	\begin{equation}\label{eq:ldp_local_lower}
	\liminf_{\kappa \to \infty} \frac{1}{\kappa} \log \mathbb{P} [ \aoc_1 \in G ] \geq -\inf_{\mu \in G} \tilde{I}(\mu).
	\end{equation}
	Moreover, $\tilde{I}$ is good, lower-semicontinuous, and convex.
	
	% Moreover, $\tI (\mu)$ is finite only if $\mu(d\z) = f(\z) d\z \ll d\z$ and equals
	% $$I(\mu) :=
	% 		\frac 12\int_{S^1} \left(\frac{d}{d\z} \sqrt{f}\right)^2  d\z,$$
	% as defined in \eqref{eq:ldp_local}.
\end{theorem}
Note that the lower-semicontinuity (hence the goodness, since $\mathcal{M}_1 (S^1)$ is compact) and convexity follow directly from the expression of $\tilde{I}$.
%  In fact, let $\mu_n$ be a sequence converging weakly to $\mu$. We have
% $$\liminf_{n \to \infty} \tI(\mu_n) \ge \sup_{u>0,u\in C^2} \liminf_{n \to \infty} -\int_{S^1} (u''/2u)\mu_n(\z) = \sup_{u>0,u\in C^2}-\int_{S^1} (u''/2u)\mu(\z) = \tI(\mu). $$
% We verify two immediate properties of $\tI$ that will be useful later.
% \begin{lemma}\label{lem-convexity}
% The rate function $\tI$ is convex and good, i.e. the sub-level sets
% \begin{equation}\label{eq-sublevel}
% \{\mu \in \M_1(S^1) \: : \:  \tI(\mu) \leq c \} \text{ for } c \geq 0
% \end{equation}
% are compact in $\M_1(S^1)$.
% \end{lemma}
% \begin{proof}
% The convexity follows from the fact that $\tI(\cdot)$ is the supremum of the linear (hence convex) functionals
% \[T_u(\mu) := -\int_{S^1} \frac{u''}{2u}(\z) \mu(d\z),\]
% where the supremum is taken over all positive $u \in C^2(S^1)$.
% Since $\tI$ is  lower-semicontinuous, the sets  $\{\mu \in \M_1(S^1) \: : \:  \tI(\mu) \leq c\}$ are closed in $\M_1(S^1)$ which is compact itself. \end{proof}
For the convenience of those readers who may not be so familiar with the statement of Theorem~\ref{thm-DV}, let us provide an outline of the proof of the upper bound \eqref{eq:ldp_local_upper} in order to explain where this rate function comes from.

%\begin{proof}[Proof of upper bound]
%First, we show that the functional $\tI$ defined in \eqref{eq:ldp_local_inf} is the large deviation rate function of $\aoc_1$.
Let $P_\zeta$ denote the law of a Brownian motion $B^1$ on $S^1$
%(with diffusivity $1$) 
starting from $\zeta \in S^1$ and $Q_{\zeta,t}$ the law of the average occupation measure $\aoc[1]_t$ under $P_\zeta$.
Fix a small number $h >0$, and let $\pi_h(\zeta, d\xi)$ be the law of $B_h$ under $P_\zeta$.
We consider the Markov chain $X_n := B_{nh}$, so that $\pi_h$ is the transition kernel of $X$. We write $\mathbb{E}$ for the expectation with respect to $P_1$.

Now let $u \in C^2 (S^1)$ such that $u > 0$.
From the Markov property, we inductively get 
\begin{align*}
\mathbb{E} \left[\frac{u(X_0) \cdots u(X_{n-1})}{\pi_h u(X_0)  \cdots \pi_h u(X_{n-1})} u(X_n)\right]
% & = \m E \left[\frac{u(X_0) \cdots u(X_{n-1})}{\pi_h u(X_0)  \cdots \pi_h u(X_{n-1})} \m E[u(X_n)\,|\,X_{n-1}]\right] \\
%& = \m  E \left[\frac{u(X_0) \cdots u(X_{n-2})}{\pi_h u(X_0)  \cdots \pi_h u(X_{n-2})} u(X_{n-1})\right] \\
& = \mathbb{E} [u(X_0)] = u(1).
\end{align*}
%The second last equality follows inductively. 
Since the Brownian motion is a Feller process with infinitesimal generator  $L$, we have
$$\log \frac{u(\zeta)}{\pi_h u (\zeta)} 
%= \log \frac{u(\z)}{ u(\z) + h Lu (\z) + o(h)}
= \log \left(1 - h \frac{Lu(\zeta)}{u(\zeta)} + o(h)\right) = - h \frac{Lu(\zeta)}{u(\zeta)} + o(h).$$
Therefore,
\begin{align*}
u(1)
%& =\m E \left[\frac{u(X_0) \cdots u(X_{n-1})}{\pi_h u(X_0)  \cdots \pi_h u(X_{n-1})} u(X_n)\right]\\
%  & =
% \m E \left[\frac{u(X_0) \cdots u(X_{n})}{\pi_h u(X_0)  \cdots \pi_h u(X_{n})} \pi_h u(X_n)\right] \\
& = \mathbb{E}\left[\exp \left(-\sum_{i = 0}^n h \frac{Lu(X_i)}{u(X_i)} + o(h)\right) \pi_h u(X_n)\right]\\
& = \mathbb{E}\left[\exp \left(-\int_{0}^{t}  \frac{Lu(B_s)}{u(B_s)} ds \right) \pi_h u(B_t) + o(1)\right],
\end{align*}
where $n$ is chosen to be the integer part of $t/h$.
Hence,
\begin{align*}
\mathbb{E}^{Q_{1,t}} \left[\exp \left(- t \int_{S^1} \frac{Lu (\zeta)}{ u(\zeta)} \aoc[1]_{t} (d\zeta) \right)\right]%&  = \m E\left[\exp \left(-\int_{0}^{t}  \frac{Lu(B_s)}{u(B_s)} ds \right)\right] \\
& \le \frac{u(1)}{ \inf_{\xi \in S^1} \pi_h u (\xi)} \le \frac{u(1)}{ \inf_{\xi \in S^1} u (\xi)} \le M(u)
\end{align*}
for some $M (u) < \infty$ depending only on the function $u>0$. %(in particular, independent of $h$ and $t$).
%Now we show the large deviation principle upper bound \eqref{eq:ldp_local_upper}. 
For any measurable set $C \subset \mathcal{M}_1(S^1)$, since
\begin{align*}
M (u)& \ge \mathbb{E}^{Q_{1,t}} \left[\exp \left(- t \int_{S^1} \frac{Lu (\zeta)}{ u(\zeta)} \aoc[1]_{t} (d\zeta) \right)\right] \\
&\ge Q_{1,t} (C) \exp \left(- t \sup_{\mu \in C} \int_{S^1} \frac{Lu}{u} (\zeta) \,\mu (d\zeta)\right)
\end{align*}
for arbitrary $u$, we have
$$\limsup_{t\to \infty} \frac{1}{t} \log Q_{1,t} (C)\le \inf_{u > 0, u\in C^2} \sup_{\mu \in C} \int_{S^1} \frac{Lu}{u} (\zeta) \,\mu (d\zeta).$$

When $C$ is closed (hence compact), some topological considerations allow us to swap the $\inf$ and $\sup$ in the above expression, and we obtain
\begin{align*}
\adjustlimits\inf_{u > 0, u\in C^2} \sup_{\mu \in C} \int_{S^1} \frac{Lu}{u} (\zeta) \mu (d\zeta) &\le  \sup_{\mu \in C} \inf_{u > 0, u\in C^2} \int_{S^1} \frac{Lu}{u} (\zeta) \mu (d\zeta) 
%\\&= -\inf_{\mu \in C} - \inf_{u > 0, u\in C^2} \int_{S^1} \frac{Lu}{u} (\z) \mu (d\z)
= -\inf_{\mu \in C} \tilde{I}(\mu),
\end{align*}
which is the upper bound \eqref{eq:ldp_local_upper}. 
As it is often the case in the derivation of large deviation principles, the lower bound turns out to be trickier, and uses approximation by discrete time Markov chains and a change of measure argument. We refer to the original paper \cite{DonVar1975} for more details.

The rate function $\tilde{I}$ of Theorem~\ref{thm-DV} is somewhat unwieldy but can be simplified for Brownian motion as noted in \cite{DonVar1975}. We provide here an alternative elementary proof.

\begin{theorem}[{\cite[Theorem 5]{DonVar1975}}]
	For $\mu\in \mathcal{M}_1(S^1)$, the rate function
	$\tilde{I}(\mu)$
	%$$\tI(\mu):=  -\inf_{ u > 0, \, u \in C^2(S^1)} \int_{S^1} \frac{u''}{2u}(\z) \,\mu (d\z)$$
	%defined in \eqref{eq:ldp_local_inf}
	is finite if and only if $\mu = \phi^2(\zeta) d\zeta$ for some function $\phi \in W^{1,2}$. In this case, we have $\tilde{I}(\mu) = I(\mu)$, where
	\[I(\mu)
	%:= \frac18 \int_{\text{supp} (f)} \frac{(f')^2}{f} \,d\z
	= \frac12 \int_{S^1} |\phi' (\zeta)|^2 \, d\zeta.\]
	%as defined in \eqref{eq:ldp_local}
\end{theorem}
\begin{proof}
	%\noindent{\bf Step 2: }
	% We now show
	% %that the rate function defined as \eqref{eq:ldp_local_inf} is the same with \eqref{eq:ldp_local} in the case of Brownian motions. That is, we show
	% $$\tI(\mu)
	% :=  -\inf_{ u > 0, \, u \in C^2(S^1)} \int_{S^1} \frac{u''}{2u}(\z) \,\mu (d\z)
	% = \frac 12\int_{S^1} \left(\frac{d}{d\z} \sqrt{f}\right)^2  \,d\z = I(\mu)$$
	% if $\mu(d\z) = f(\z) d\z$ where $\sqrt{f}$ is some absolutely continuous function on $S^1$, or the LHS is infinite otherwise.
	First assume that $\mu = \phi^2 d\zeta$ for some $\phi \in W^{1,2}$ and that $I(\mu)$ is finite, we will show that $\tilde{I} (\mu) = I(\mu)$.
	%First assume that $\mu(d\z) = f\,d\z$ where $\phi:=\sqrt{f}$ on $S^1$ is absolutely continuous that makes $I(\mu)$ finite. Since $\phi\in W^{1,2}$, t
	%Writing $\phi = \sqrt f$, t
	For this, take a sequence $\phi_n\in C^2$ with $\phi_n>0$ converging to $\phi$ almost everywhere such that $\int_{S^1}(\phi_n')^2\,d\zeta\rightarrow \int_{S^1}(\phi')^2\,d\zeta$. For any $u\in C^2$ and  any $\varepsilon>0$, we have for sufficiently large $n$ that
	$$\int_{S^1} \frac{u''}{2u}\phi^2\,d\zeta+\varepsilon \ge \int_{S^1} \frac{(v\phi_n)''}{2(v\phi_n)}\phi_n^2\,d\zeta = \int_{S^1}\frac{\phi_n''\phi_n}{2}\,d\zeta + \int_{S^1}\frac{(\phi_n^2v')'}{2v}\,d\zeta,$$
	where $v:=u/\phi_n\in C^2$.
	From integration by parts, this latter expression equals to
	$$-\frac{1}{2}\int_{S^1}|\phi_n'|^2\, d\zeta + \frac{1}{2}\int_{S^1}\frac{\phi_n^2v'^2}{v^2}\,d\zeta \ge -\frac{1}{2}\int_{S^1}|\phi_n'|^2\, d\zeta \ge -I(\mu)-\varepsilon$$
	by taking $n$ larger if necessary. Since $\varepsilon$ is arbitrary, we obtain $-\int_{S^1} \frac{u''}{2u}\phi^2\,d\zeta \leq I(\mu)$, and thus $\tilde{I}(\mu)\le I(\mu)$ by taking supremum over $u$. The opposite inequality
	%$\tI(\mu)\ge \frac{1}{2}\int_{S^1}(\phi')^2\,d\z$
	can be shown by taking $u=\phi_n$ (i.e. $v=1$) and sending $n \to \infty$. Therefore $\tilde{I}(\mu) = I(\mu)$ when $I(\mu) < \infty$.
	
	It remains to prove that if $\tilde{I} (\mu) <\infty$ then $I(\mu) <\infty$, so consider $\mu$ such that $\tilde{I}(\mu)$ is finite.
	%Let $\eta$ be a positive mollifier. In other words $\eta\ge 0$ is smooth and supported in a small neighborhood of $1$ with $\int_{S^1} \eta(\z)\, d\z=1$, and
	Let $\{\eta_\varepsilon\}_{\varepsilon > 0}$ be a family of nonnegative smooth functions on $S^1$ with $\int_{S^1}\eta_\varepsilon\, d\zeta = 1$ and converging weakly to the Dirac delta function at $1$ as $\varepsilon\to 0$.
	%For each $\eps>0$, define a function $f_\eps = \eta_\eps * \mu$ and a measure
	%\[\mu_\eps(d\z) := f_\eps(\z)d\z = \int_{S^1} \eta_\eps(\z/\xi)\,\mu(d\xi)\, d\z\]
	Writing $\mu^\xi$ for $\mu$ rotated by $\xi \in S^1$, we define $\mu_\varepsilon = \int_{S^1} \eta_\varepsilon(\xi) \mu^{\xi} d\xi$ as a weighted average of probability measures so that $\mu_\varepsilon$ converges weakly to $\mu$.  Observe that $\tilde{I}$ is rotation invariant and convex.
	Therefore by Jensen's inequality,
	%\[\tI(\mu_\eps) = \tI(\int_{S^1}\eta_\eps(\z/\xi)\,\mu(d\xi))\le \int_{S^1} \eta_\eps(\z) \tI(\mu)\, d\z = \tI(\mu).\]
	\[\tilde{I}(\mu_\varepsilon) = \tilde{I}\left(\int_{S^1}\eta_\varepsilon(\xi)\mu^\xi\, d\xi \right)\le \int_{S^1} \eta_\varepsilon(\xi) \tilde{I}(\mu^\xi)\, d\xi = \tilde{I}(\mu).\]
	
	Write $\phi_\varepsilon := \sqrt{\eta_\varepsilon * \mu}$, so that $\mu_\varepsilon = \phi_\varepsilon^2(\zeta) d\zeta$. 
	We will show that 
	\begin{equation}\label{eq-dirichlet-energy-mollified}
	\int_{S^1} (\phi_\varepsilon')^2\,d\zeta\le 2\tilde{I}(\mu_\varepsilon).
	\end{equation}
	Given~\eqref{eq-dirichlet-energy-mollified}, we see that $\int_{S^1} (\phi_\varepsilon')^2\,d\zeta$ is uniformly bounded above (by $2\tilde{I}(\mu)$);
	%Now we claim $\phi_\eps \in W^{1,2}$ and $\int_{S^1} (\phi_\eps')^2\,d\z\le 2\tI(\mu_\eps)$ which is uniformly bounded by $2\tI(\mu)<\infty$. 
	letting $\varepsilon\to 0$  implies $\mu$ is an absolutely continuous  measure, and furthermore $\sqrt{\mu(d\zeta)/d\zeta}\in W^{1,2}$ so $I(\mu) < \infty$, concluding the proof of the theorem.
	
	We turn to the proof of~\eqref{eq-dirichlet-energy-mollified}, which follows 
	the argument of~\cite[Lemma 2.4]{donsker1980law}. In the definition \eqref{eq:ldp_local_inf}, take $u=e^{\lambda h}$ where $h$ is smooth and $\lambda$ is a real number. This gives
	\[\lambda^2\int_{S^1}h'^2\phi_\varepsilon^2\,d\zeta+\lambda\int_{S^1}h''\phi_\varepsilon^2\,d\zeta
	=\lambda^2\int_{S^1}h'^2\phi_\varepsilon^2\,d\zeta-2\lambda\int_{S^1}h'\phi'_\varepsilon\phi_\varepsilon\,d\zeta
	\ge -2\tilde{I}(\mu_\varepsilon)\]
	which holds for any real number $\lambda$. By choosing $\lambda$ so that the quadratic function takes the minimum, we have
	\begin{equation}\label{eq:opt-ineq}\left(\int_{S^1} h'\phi_\varepsilon'\phi_\varepsilon\,d\zeta\right)^2\le 2\tilde{I}(\mu_\varepsilon)\int_{S^1} h'^2\phi_\varepsilon^2\,d\zeta.
	\end{equation}
	For $n\in \mathbb{N}$, consider an auxiliary function $\nu_n$ on positive real numbers defined as
	$$
	\nu_n(x) := \begin{cases}
	0 & \text{if } 0<x\le 1/2n\\
	1/x & \text{if } x\ge 1/n
	\end{cases}
	$$
	and extended on $[1/2n,1/n]$ so that $0\le \nu_n(x)\le 1/x$ and $\nu_n$ is smooth for all $x$. And define $V_n(x)=\int_0^x \nu_n(y)\, dy$. Plugging $h=V_n(\phi_\varepsilon)$ to \eqref{eq:opt-ineq} gives
	$$\left(\int_{S^1} \nu_n(\phi_\varepsilon) \phi_\varepsilon'^2\phi_\varepsilon \,d\zeta\right)^2 \le 2\tilde{I}(\mu_\varepsilon) \int_{S^1} \nu_n^2(\phi_\varepsilon)\phi_\varepsilon'^2\phi_\varepsilon^2\,d\zeta
	\le 2\tilde{I}(\mu_\varepsilon)\int_{S^1} \nu_n(\phi_\varepsilon)\phi_\varepsilon'^2\phi_\varepsilon\, d\zeta$$
	where $\nu_n(\phi_\varepsilon)\le 1/\phi_\varepsilon$ was used and the common terms on both sides cancel out. As $n\to \infty$, Fatou's lemma implies~\eqref{eq-dirichlet-energy-mollified} as desired. \qedhere

\end{proof}

%For the radial SLE with driving function given by Brownian motion $B^{\k}_t$ on $S^1$ with diffusivity $\sqrt {\k}$, let $K_t$ be the hull and $U^{\k}_t := \mathbb D \backslash K_t$ be the unexplored region. Define $g^{\k}_t: U_t \to \mathbb D$ to be the mapping out function; i.e. the unique conformal map $g$ sending $0 \mapsto 0$ and satisfying $g'(0) > 0$. In particular, $g^{\k}_0$ is the identity function. Fix $T > 0$, we show that the family of functions $(g^{\k}_t)_{t \in [0,T]}$ converges to the functions $(z \mapsto e^t z)_{t \in [0,T]}$ in the sense that for any compact set $K \subset e^{-T} \mathbb D$, we have the uniform convergence $g^{\k}_t(z) \to e^t z$ (uniformly for all $z \in K$ and $t \in [0,T]$) in law as $\k \to \infty$.

\subsection{Large deviations for $\{\delta_{\bm_t}\}$}\label{sec:proof_1.2}
In this section, we prove Theorem~\ref{thm:ldp-circ-bm}. That is, we establish the large deviation principle for the Brownian trajectory measure $\{\delta_{\bm_t}\}$. We use the notation of Section~\ref{section-approx}.

The first step is the large deviation principle for $P_n(\{\delta_{\bm_t}\})$, which follows easily from Theorem~\ref{thm-DV}. Recall that $P_n(\{\delta_{\bm_t}\})$ is a $2^n$-tuple of elements of $\mathcal{M}_1(S^1)$, the $i$th element being the average of $\{\delta_{\bm_t}\}$ on the time interval $[i/2^n, (i+1)/2^n]$.
\begin{lemma}\label{lem-LDP-proj}
	Fix $n \ge 1$. The random variable $P_n(\{\delta_{\bm_t}\}) \in \mathcal{Y}_n$ satisfies the large deviation principle as $\kappa \to \infty$, with good rate function $I_n: \mathcal{Y}_n \to \mathbb{R}$ defined by
	\begin{equation}\label{eq-In}
	I_n((\mu^i)_{i \in \mathcal{I}_n}) := \frac{1}{2^n} \sum_{i \in \mathcal{I}_n} I(\mu_i),
	\end{equation}
	where $I: \mathcal{M}_1(S^1) \to \mathbb{R}$ is the good rate function defined in~\eqref{eq:ldp_local}.
\end{lemma}
\begin{proof}
	Since the large deviation rate function $I$ is rotation invariant, the same rate function is applicable to the setting of the occupation measure of Brownian motion started at any $\zeta \in S^1$. Furthermore, the Markov property of Brownian motion tells us that conditioned on the value $B^\kappa_{j/2^n}$, the process $\left(B^\kappa_t\right)_{[j/2^n, (j+1)/2^n]}$ is independent of $\left(B^\kappa_t\right)_{[0, j/2^n]}$. These observations, together with Theorem~\ref{thm-DV}, yield the lemma.
	%TODO. The rate function can be shown to be correct using the contraction principle. For the paper I think we should omit the proof and claim it obvious (I think it should be very obvious to people versed in LDP, I am not though). For completeness here's a proof sketch: Consider the space $X_1$ comprising independent $2^n$-tuples of circular $\k$-Brownian motion run for time $2^{-n}$ each, each started at 0. If we concatenate these appropriately we get a single BM $B$ run for time 1. Map to the space $X_2$ of $2^n$-tuples of the occupation measures of the single BM $B$ on each of the intervals $[i/2^n, (i+1)/2^n]$. This mapping is continuous, and the LDP for occupation measure doesn't depend on the start point, so we get the above rate function by the contraction principle in the map $X_1 \to X_2$.
\end{proof}

Since $\mathcal{N} = \varprojlim \mathcal{Y}_n$, we can deduce the large deviation principle for $\{\delta_{\bm_t}\}$.
\begin{proposition}\label{prop-trajectory-ldp}
	The random measure $\{\delta_{\bm_t}\} \in \mathcal{N}$ has the large deviation principle with good rate function
	\[\sup\nolimits_{n \geq 0} I_n(P_n(\rho)) \quad \text{ for } \rho \in \mathcal{N}, \]
	where $I_n: \mathcal{Y}_n \to \mathbb{R}$ is defined in \eqref{eq-In}.
\end{proposition}
\begin{proof}
	This follows from the Dawson-G\"artner theorem~\cite{DawGae1987} (or {\cite[Thm 4.6.1]{DemZei2010}}), the fact that $\mathcal{N} = \varprojlim \mathcal{Y}_n$ by Lemma~\ref{lem-projection}, and the large deviation principle for $P_n(\{\delta_{\bm_t}\})$ (Lemma~\ref{lem-LDP-proj}).
\end{proof}

Finally, we can simplify the rate function $\sup_{n\ge 0} I_n(P_n(\rho))$.

\begin{lemma}\label{lem-energy}
	Define $\mathcal{E}: \mathcal{N} \to \mathbb{R}$ by
	$\mathcal{E}(\rho) := \int_0^1 I(\rho_t)\,dt,$
	where $\{\rho_t\}$ is any disintegration of $\rho$ with respect to $t$ (see \eqref{def:disint}). Then, with $I_n: \mathcal{Y}_n \to \mathbb{R}$ defined as in \eqref{eq-In}, we have
	\[\mathcal{E}(\rho) = \sup\nolimits_{n\ge 0} I_n(P_n(\rho)).\]
\end{lemma}
\begin{proof}
	By definition we have $P_{n}^{i}(\rho) = \frac12(P_{n+1}^{2i}(\rho) + P_{n+1}^{2i+1}(\rho))$, so Jensen's inequality applied to the convex function $I$ yields
	$I_n(P_n(\rho)) \leq I_{n+1}(P_{n+1}(\rho))$, 
	and hence
	\[
	\sup\nolimits_{n\ge 0} I_n(P_n(\rho)) = \lim_{n \to \infty} I_n(P_n(\rho)).
	\]
	Next, we check that $\mathcal{E}(\rho) \geq \lim_{n \to \infty} I_n(P_n(\rho))$. This again follows from Jensen's inequality:
	\[I_n(\rho) = \frac{1}{2^n} \sum_{i \in \mathcal{I}_n} I\left( 2^n \int_{i/2^n}^{(i+1)/2^n} \rho_t \ dt  \right) \leq \sum_{i \in \mathcal{I}_n} \int_{i/2^n}^{(i+1)/2^n} I\left(  \rho_t  \right) \ dt = \mathcal{E}(\rho).  \]
	Thus, we are done once we prove the reverse inequality $\mathcal{E}(\rho) \leq \lim_{n \to \infty} I_n(P_n(\rho))$.
	
	Consider the probability space given by $[0,1]$ endowed with its Borel $\sigma$-algebra $\mathcal{F}_\infty$, and let $\mu$  be the $\mathcal{M}_1(S^1)$-valued random variable defined by sampling $t \sim \operatorname{Leb}_{[0,1]}$ then setting $\mu := \rho_t$. Let $\mathcal{F}_n$ be the $\sigma$-algebra generated by sets of the form $[i/2^n, (i+1)/2^n]$ for $i \in \mathcal{I}_n$; note that $\mathcal{F}_\infty = \sigma(\cup_n \mathcal{F}_n)$. Define
	$\mu_n:= \mathbb E \left[\mu \big| \mathcal{F}_n \right].$
	For any continuous function $f \in C(S^1)$, the bounded real-valued Doob martingale $\mu_n(f)$ converges a.s. to $\mu(f)$. Taking a suitable countable collection of $f$, we conclude that a.s. $\mu_n$ converges to $\mu$ in the Prokhorov topology. By Fatou's lemma and the lower-semicontinuity of $I$, we have
	\begin{equation}\label{eq-fatou}
	\liminf_{n \to \infty} \mathbb E [ I(\mu_n)] \geq \mathbb E [ \liminf_{n \to \infty} I(\mu_n)] \geq \mathbb E [ I(\mu)].
	\end{equation}
	We can write $\mu_n$ explicitly as $\mu_n = 2^n \int_{i/2^n}^{(i+1)/2^n} \rho_t \ dt$ where $i \in \mathcal{I}_n$ is the index for which $t \in [i/2^n, (i+1)/2^n]$, so
	$\mathbb E[I(\mu_n)] = I_n(P_n(\rho)). $
	We also have
	$\mathbb E[I(\mu)] = \int_0^1 I(\rho_t) \ dt = \mathcal{E}(\rho). $
	Combining these with~\eqref{eq-fatou}, we conclude that $\lim_{n\to\infty} I_n(P_n(\rho)) \geq \mathcal{E}(\rho)$.
\end{proof}

%We are now ready to prove Theorem~\ref{thm:ldp-circ-bm}.
\begin{proof}[Proof of Theorem~\ref{thm:ldp-circ-bm}]
	Proposition~\ref{prop-trajectory-ldp} says that $\{\delta_{\bm_t}\} \in \mathcal{N}$ has a large deviation principle with good rate function given by $\sup_{n\geq0} I_n(P_n(\cdot))$, and Lemma~\ref{lem-energy} shows that this good rate function can alternatively be expressed as $\mathcal{E}$.
	%The theorem follows immediately by combining Proposition~\ref{prop-trajectory-ldp} and Lemma~\ref{lem-energy}. 
\end{proof} 
\section{Comments}\label{section-comments}
Let us make further comments and list a few questions in addition to those in the introduction.
\begin{enumerate}[wide, labelwidth=!, labelindent=0pt]
	\item
	As we have discussed in the introduction, one may wonder what the limit and large deviations of chordal $\operatorname{SLE}_{\infty}$ are. 
	Figure~\ref{fig:sle-large-kappa} shows two chordal $\operatorname{SLE}_\kappa$ curves on $[-1,1]^2$ from a boundary point $-i$ to another boundary point $i$, for several large values of $\kappa$. We see that the interfaces stretch out to the target point and are close to horizontal lines after we map the square to $\mathbb{H}$ and the target point $i$ to $\infty$.
	\begin{figure}[ht]
		\begin{subfigure}{.4\textwidth}
			\centering
			\includegraphics[width=\textwidth]{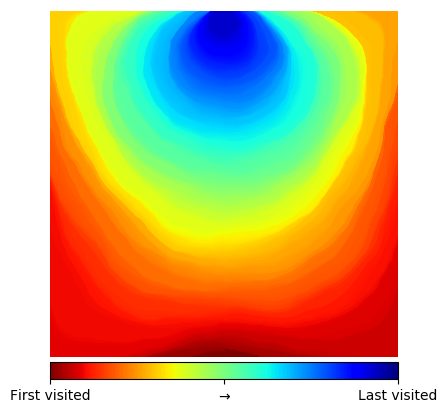}
		\end{subfigure}
		\centering
		\begin{subfigure}{.4\textwidth}
			\centering
			\includegraphics[width=\textwidth]{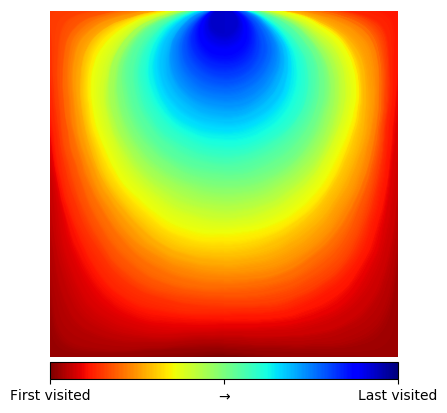}
		\end{subfigure}
		\caption{An instance of chordal $\operatorname{SLE}_{128}$ and $\operatorname{SLE}_{1000}$ on $[-1,1]^2$ from $-i$ to $i$. The simulation of these \emph{counterflow lines} is done by imaginary geometry as described in \cite{IG1}, and are approximated via linear interpolation of an $800\times 800$ discrete Gaussian free field with suitable boundary conditions. The color represents the time (capacity) parametrization of the $\operatorname{SLE}$ curve.}\label{fig:sle-large-kappa}
	\end{figure}
	
	\item Corollary~\ref{cor:ldp-sle} shows that $\operatorname{SLE}_{\infty}$ concentrates around the family of Loewner chains driven by an absolutely continuous measure $\rho$ with $\mathcal{E}(\rho) < \infty$. In \cite{VW2} we geometrically characterize the Loewner chains driven by such measures. 
	Note that the answer to the same question for the large deviation rate function of $\operatorname{SLE}_{0+}$, namely
	the family of Jordan curves of finite Loewner energy, is well-understood. That family has been shown to be exactly the family of Weil-Petersson quasicircles \cite{W2}, which has far-reaching connections to geometric function theory and Teichm\"uller theory. 
	\item \label{comment-rate}  The rate function \eqref{eq:ldp_local} for the Brownian occupation measure coincides with the rate function of the \emph{square} of the Brownian bridge (or Gaussian free field) on $S^1$. Is there a profound reason or is this merely a coincidence? One could attempt to relate the large deviations of the Brownian occupation measure to the large deviations of the occupation measure of a \emph{Brownian loop soup} on $S^1$.  
	
	\item The fluctuations of the circular Brownian occupation measure $\oc_t$ were studied by Bolthausen. We express this result in terms of the \emph{local time} $\ell_t : S^1 \to [0,\infty)$, defined via $\oc[1]_t = \ell_t(\zeta)d \zeta$. Note that $\ell_t$ is a.s. a random continuous function.
	
	\begin{theorem}[{\cite{Bol1979}}]\label{thm-Bolthausen}
		Identify each $\zeta=e^{i\theta}\in S^1$ with $\theta\in [0,2\pi)$. As $t \to \infty$, the stochastic process $\sqrt t (\frac{\ell_t(\theta)}{t} - \frac{1}{2\pi})_{\theta \in [0,2\pi)}$ converges in distribution to 
		%twice the mean-centered Brownian bridge process on $S^1$. Explicitly, this limit law is given by 
		$(2b_\theta - \frac1\pi \int_0^{2\pi} b_{\tau}  \ d \tau )_{\theta \in [0,2\pi)}$, where $b$ is a Brownian bridge on the interval $[0, 2\pi]$ with endpoints pinned at $b_0 = b_{2\pi} = 0$.
	\end{theorem}
	%The fact that the fluctuation of the Brownian occupation measure in Theorem~\ref{thm-Bolthausen} is \emph{twice} the Brownian bridge is consistent with the \emph{square} in Comment~\ref{comment-rate}. 
	%Notice that we have not fully explored the knowledge about the fluctuations and i
	We wonder whether there are interesting consequences to the fluctuations of $\operatorname{SLE}_{\infty}$.
\end{enumerate}
\input{tex/bibliography.tex}
\end{document}

%% file: macros.tex
\DeclareDocumentCommand{\oc}{o}{L\IfValueTF{#1}{^{#1}}{^\kappa}}
\DeclareDocumentCommand{\aoc}{o}{\overline{L}\IfValueTF{#1}{^{#1}}{^\kappa}}
\DeclareDocumentCommand{\bm}{o}{B\IfValueTF{#1}{^{#1}}{^\kappa}}
\DeclareDocumentCommand{\dbm}{o}{\beta\IfValueTF{#1}{^{#1}}{^\kappa}}
\DeclarePairedDelimiterX{\expectarg}[1]{[}{]}{%
  \ifnum\currentgrouptype=16 \else\begingroup\fi
  \activatebar#1
  \ifnum\currentgrouptype=16 \else\endgroup\fi
}
\newcommand{\innermid}{\nonscript\;\delimsize\vert\nonscript\;}
\newcommand{\activatebar}{%
  \begingroup\lccode`\~=`\|
  \lowercase{\endgroup\let~}\innermid
  \mathcode`|=\string"8000
}

%% file: tex/title.tex
\title{ Large deviations of radial SLE$_{\infty}$}
% \author[]{Morris Ang}
% \author[]{Minjae Park}
% \author[]{Yilin Wang}
% \affil[]{Massachusetts Institute of Technology}
% \date{Preliminary draft}

\makeatletter
\renewcommand\@date{{%
  \vspace{-\baselineskip}%
  \large\centering
  \begin{tabular}{@{}c@{}}
    Morris Ang \\
    \normalsize angm@mit.edu
  \end{tabular}%
  \qquad
  \begin{tabular}{@{}c@{}}
    Minjae Park \\
    \normalsize minj@mit.edu
  \end{tabular}
  \qquad
  \begin{tabular}{@{}c@{}}
    Yilin Wang \\
    \normalsize yilwang@mit.edu
  \end{tabular}

  \textit{Massachusetts Institute of Technology}\par
  
  \vspace{\baselineskip}
  
 \today
}}
\makeatother
\maketitle

\begin{abstract}
We derive the large deviation principle for radial Schramm-Loewner evolution ($\operatorname{SLE}$) on the unit disk with parameter $\kappa \rightarrow \infty$. Restricting to the time interval $[0,1]$, the good rate function is finite only on a certain family of Loewner chains driven by absolutely continuous probability measures $\{\phi_t^2 (\zeta)\, d\zeta\}_{t \in [0,1]}$ on the unit circle and equals $\int_0^1 \int_{S^1} |\phi_t'|^2/2\,d\zeta \,dt$. Our proof relies on the large deviation principle for the long-time average of the Brownian occupation measure by Donsker and Varadhan.
\end{abstract}

%% file: tex/bibliography.tex
\bibliographystyle{hamsalpha}
\bibliography{sle-large-deviation}